\documentclass[a4paper,11pt,leqno,twoside]{article}

\usepackage{amsmath, amsthm, amssymb,bm}

\numberwithin{equation}{section}

\newcommand{\bP}{\mathbb{P}}

\newcommand{\bR}{\mathbb{R}}

\newcommand{\bZ}{\mathbb{Z}}

\newcommand{\mf}[1]{\mathfrak{#1}}
\newcommand{\mr}[1]{\mathrm{#1}}
\newcommand{\mcal}[1]{\mathcal{#1}}

\def\tr{\mathrm{tr}}

\def\({ \left( }
\def\){ \right)}

\def\GBE#1{\mathrm{G{#1}E}}
\def\GUE{\mathrm{GUE}}
\def\GOE{\mathrm{GOE}}
\def\GSE{\mathrm{GSE}}

\theoremstyle{plain}
\newtheorem{thm}{Theorem}[section]
\newtheorem{prop}[thm]{Proposition}
\newtheorem{lem}[thm]{Lemma}
\newtheorem{cor}[thm]{Corollary}
\theoremstyle{definition}

\theoremstyle{conjecture}

\theoremstyle{problem}

\title{\bfseries Jack deformations of Plancherel measures  and
traceless Gaussian random matrices}
\author{\textsc{Sho MATSUMOTO} \\
 {\it Graduate School of Mathematics, Nagoya University.} \\
{\it  \small Furocho, Chikusa-ku, Nagoya, 
464-8602, JAPAN.} \\
\texttt{sho-matsumoto@math.nagoya-u.ac.jp}
}

\date{\empty}

\begin{document}
%>>>>>>>>>>>>>>>>>>>>>>>>>>>>>>>>>>>>>>>>>
%\setlength{\baselineskip}{15pt}
\maketitle

\begin{abstract}
We study random partitions $\lambda=(\lambda_1,\lambda_2,\dots,\lambda_d)$ of $n$
whose length is not bigger than a fixed number $d$.
Suppose a random partition $\lambda$ is distributed according to 
the Jack measure, which is
a deformation of the Plancherel measure with a positive parameter $\alpha>0$.
We prove that for all $\alpha>0$,
in the limit as $n \to \infty$,
the joint distribution of scaled $\lambda_1,\dots, \lambda_d$ converges 
to the joint distribution of some random variables from a traceless Gaussian $\beta$-ensemble 
with $\beta=2/\alpha$.
We also give a short proof of Regev's asymptotic theorem
for the sum of $\beta$-powers of $f^\lambda$, the number of standard tableaux of shape $\lambda$.

\noindent
{\bf MSC-class}: primary 60C05 ; secondary 05E10 
\end{abstract}

%((((((((((((((    Main Text follows    ))))))))))))))))

%
%%%%%%%%%%%%%%%%%%%%%%%%%%%%%%%%%%%%%%%%%%%%%%%%%%%%%%%%%%%%%%%%%%%%%%%%%%%%%%%%%%%%%%%%%%%%%%%%
\section{Introduction}
%%%%%%%%%%%%%%%%%%%%%%%%%%%%%%%%%%%%%%%%%%%%%%%%%%%%%%%%%%%%%%%%%%%%%%%%%%%%%%%%%%%%%%%%%%%%%%%%
%

A random partition is studied as a discrete analogue of eigenvalues of a random matrix.
The most natural and studied random partition is a partition distributed according to
the Plancherel measure for the symmetric group.
The Plancherel measure chooses a partition $\lambda$ of $n$ with probability
\begin{equation} \label{eq:Plancherel}
\bP^{\mr{Plan}}_n(\lambda)= \frac{(f^\lambda)^2}{n!},
\end{equation}
where $f^\lambda$ is the degree of the irreducible representation of the symmetric group $\mf{S}_n$
associated with $\lambda$.
A random partition $\lambda=(\lambda_1,\lambda_2,\dots)$ 
chosen by the Plancherel measure is 
closely related to the Gaussian unitary ensemble (GUE) of random matrix theory.
The GUE matrix is a Hermitian matrix whose entries are independently distributed 
according to the normal distribution.
The probability density function for the eigenvalues $x_1 \ge \dots \ge x_d$
of the $d \times d$ GUE matrix is proportional to
\begin{equation} \label{eq:GBEdensity}
e^{-\frac{\beta}{2}(x_1^2+ \cdots +x_d^2)} \prod_{1 \le i<j \le d} (x_i-x_j)^\beta
\end{equation}
with $\beta=2$.
In \cite{BOO, Johansson, Okounkov1} (see also \cite{BDJ}), it is proved that, as $n \to \infty$, 
the joint distribution of
the scaled random variables $(\lambda_i-2\sqrt{n})n^{-1/6}$, $i=1,2,\dots,$
according to $\bP^{\mr{Plan}}_{n}$
converges to 
a distribution function $F$. 
Meanwhile, 
the joint distribution of
the scaled eigenvalues $(x_i-\sqrt{2d})\sqrt{2} d^{1/6}$
of a $d \times d$ GUE matrix converges to the same function $F$ as $d \to \infty$
(\cite{TWlevel}).
Thus, roughly speaking,
a limit distribution for $\lambda_i$ in $\bP^{\mr{Plan}}_n$
equals a limit distribution for eigenvalues $x_i$ of a GUE random matrix.

An analogue of the Plancherel measure on strict partitions (i.e., all non-zero $\lambda_i$
are distinct each other), called the shifted Plancherel measure,
is studied in \cite{MatSS1, MatSS2}, see also \cite{TW2}.
It is proved that the joint distribution of scaled $\lambda_i$ of 
the corresponding random partition also converges to the limit distribution
for a GUE matrix.
In addition, there are many recent works 
(\cite{Borodin, BorodinOlshanskiStrahov, 
JohanssonUnitary, JohanssonShape,  KerovBook, Okounkov2, TW1}),
which evinces the connection between 
Plancherel random partitions and
GUE random matrices.

In random matrix theory, there are two much-studied analogues of the GUE matrix,
called the Gaussian orthogonal (GOE) and symplectic (GSE) ensemble random matrix,
see standard references \cite{Forrester, Mehta}.
The probability density function for the eigenvalues of the GOE and GSE matrix is
proportional to the function given by \eqref{eq:GBEdensity} with $\beta=1$ and $\beta=4$,
respectively.
It is natural to consider a model of random partitions 
corresponding to the GOE and GSE matrix.
This motivation is not new and one may recognize it in \cite{BR1, BR2, BR3, FNR, FR}.
In the present paper, we deal with
a ``$\beta$-version'' of the Plancherel measure,
called 
the {\it Jack measure} with parameter $\alpha:=2/\beta$ 
(\cite{BorodinOlshanski, FulmanJack, FulmanCharacter, Okounkov2, Strahov}).

The Jack measure with a positive real parameter $\alpha>0$ equips
to each partition $\lambda$ of $n$ the probability
$$
\mathbb{P}^{\mr{Jack},\alpha}_n(\lambda)= 
\frac{\alpha^n n!}{c_{\lambda}(\alpha) c'_{\lambda}(\alpha)}.
$$
Here $c_{\lambda}(\alpha)$ and $c'_{\lambda}(\alpha)$ are defined by \eqref{eq:JackHook} below
and are $\alpha$-analogues of the hook-length product of $\lambda$.
We notice that
the Jack measure with parameter $\alpha=1$ agrees the Plancherel measure $\bP^{\mr{Plan}}_n$
because $\frac{n!}{c_{\lambda}(1)}=\frac{n!}{c'_{\lambda}(1)}=f^\lambda$.
One may regard a random partition distributed according to
the Jack measure with parameter $\alpha=2$ and $\alpha=1/2$
as a discrete analogue of the GOE and GSE matrix, respectively.
More generally, for any positive real number $\beta>0$, 
the Jack measure with $\alpha=2/\beta$ is the counterpart of the Gaussian $\beta$-ensemble 
(G$\beta$E)
with the probability density function proportional to \eqref{eq:GBEdensity}.

We are interested in finding out an explicit connection between Jack measures and 
the G$\beta$E.
In the present paper, we deal with  random partitions
with at most $d$ non-zero $\lambda_j$'s, where $d$ is a fixed positive integer.
Let $\mcal{P}_n(d)$ be the set of such partitions of $n$, i.e., 
$\lambda \in \mcal{P}_n(d)$ is a weakly-decreasing $d$-length sequence 
$(\lambda_1,\dots, \lambda_d)$ of non-negative integers such that $\lambda_1+ \cdots +\lambda_d=n$.
Let  $\lambda^{(n)}=(\lambda_1^{(n)}, \dots, \lambda_d^{(n)})$
 be a random partition in $\mcal{P}_n(d)$ chosen with 
the Jack measure.
Then, for each $1 \le i \le d$, the function $\lambda^{(n)} \mapsto \lambda_i^{(n)}$ 
defines a random variable on $\mcal{P}_n(d)$.
\'{S}niady \cite{Sniady} proved that, if $\alpha=1$ (the Plancherel case),
 the joint distribution of the random variables
$\(\sqrt{\frac{d}{n}} (\lambda_i^{(n)} -\frac{n}{d})\)_{1 \le i \le d}$ converges,
as $n \to \infty$, to the joint distribution of the eigenvalue of a $d \times d$ 
traceless GUE matrix.
(Note that our definition of the probability density function \eqref{eq:GBEdensity} 
with $\beta=2$ is slightly different from \'{S}niady's one.)
Here the traceless GUE matrix is a GUE matrix whose trace is zero.

Our goal in the present paper is to extend \'{S}niady's result to 
Jack measures with any parameter $\alpha$.
Specifically, 
let  a random partition $\lambda^{(n)} \in \mcal{P}_n(d)$ 
to be chosen in the Jack measure.
Then, we prove that
the joint distribution of the random variables 
$\(\sqrt{\frac{\alpha d}{n}} (\lambda_i^{(n)} -\frac{n}{d})\)_{1 \le i \le d}$ converges
to the joint distribution of eigenvalues in the traceless $\GBE{\beta}$ with $\beta=2/\alpha$.
The explicit statement of our main result is given in \S \ref{sec:mainresult}
and its proof is given in \S \ref{sec:Proof}.

In \S \ref{sec:RSK}, we focus on Jack measures with $\alpha=2$ and $\alpha=\frac{1}{2}$.
These are discrete analogues of GOE and GSE random matrices.
Via the RSK correspondence
between permutations and pairs of standard Young tableaux, 
we see connections with random involutions.

In the final section \S \ref{sec:Regev}, 
we give a short proof of Regev's asymptotic theorem.
Regev \cite{Regev} gave an asymptotic behavior for the sum 
$$
\sum_{\lambda \in \mcal{P}_n(d)} (f^\lambda)^\beta
$$
in the limit $n \to \infty$.
In this limit value, the normalization constant of the traceless $\GBE{\beta}$ appears.
Regev's asymptotic theorem is an important classical result which indicates
a connection between Plancherel random partitions and random matrix theory.
Applying the technique used in the proof of our main result,
we obtain a short proof of Regev's asymptotic theorem.

Throughout this paper, we let $d$ to be a fixed positive integer.

%
%%%%%%%%%%%%%%%%%%%%%%%%%%%%%%%%%%%%%%%%%%%%%%%%%%%%%%%%%%%%%%%%%%%%%%%%%%%%%%%%%%%%%%%%%%%%%%%%
\section{Main result} \label{sec:mainresult}
%%%%%%%%%%%%%%%%%%%%%%%%%%%%%%%%%%%%%%%%%%%%%%%%%%%%%%%%%%%%%%%%%%%%%%%%%%%%%%%%%%%%%%%%%%%%%%%%
%

\subsection{Jack measures with parameter $\alpha>0$}

We review  fundamental notations for partitions according to \cite{Sagan, Mac}.
A partition $\lambda=(\lambda_1,\lambda_2,\dots)$ 
is a weakly decreasing sequence of non-negative integers
such that $\lambda_j=0$ for $j$ sufficiently large.
Put
$$
\ell(\lambda) =\# \{ j \ge 1  \ | \ \lambda_j>0\}, \qquad
|\lambda|=\sum_{j \ge 1} \lambda_j
$$
and call them the length and weight of $\lambda$, respectively.
If $|\lambda|=n$, we say that $\lambda$ is a partition of $n$.
We identify $\lambda$ with the corresponding Young diagram 
$\{(i,j) \in \bZ^2 \ | \ 1 \le i \le \ell(\lambda), \ 1 \le j \le \lambda_i\}$.
We write  $(i,j) \in \lambda$ if $(i,j)$ is contained in the Young diagram of $\lambda$.
Denote by $\lambda'=(\lambda'_1,\lambda'_2,\dots)$ 
the conjugate partition of $\lambda$, i.e.,
$(i,j) \in \lambda'$ if and only if $(j,i) \in \lambda$.

Let $\alpha$ be a positive real number.
For each partition $\lambda$, 
we put
(\cite[VI. (10.21)]{Mac})
\begin{equation} \label{eq:JackHook}
c_{\lambda}(\alpha)= \prod_{(i,j) \in \lambda} (\alpha(\lambda_i-j)+(\lambda_j'-i)+1), \qquad
c'_{\lambda}(\alpha)= \prod_{(i,j) \in \lambda} (\alpha(\lambda_i-j)+(\lambda_j'-i)+\alpha).
\end{equation}
Let $\mcal{P}_n$ be the set of all partitions of $n$.
Define
\begin{equation} \label{eq:defJackMeasure}
\bP^{\mr{Jack},\alpha}_{n}(\lambda)= \frac{ \alpha^n n!}{c_{\lambda}(\alpha) c'_{\lambda}(\alpha)}
\end{equation}
for each $\lambda \in \mcal{P}_n$.
This is a probability measure on $\mcal{P}_n$, i.e., 
$\sum_{\lambda \in \mcal{P}_n}\bP^{\mr{Jack},\alpha}_{n}(\lambda) =1$,
see \cite[VI (10.32)]{Mac}.
We call this the {\it Jack measure with parameter $\alpha$} (\cite{FulmanJack, FulmanCharacter}).
This is sometimes called the Plancherel measure with parameter $\theta :=\alpha^{-1}$
(\cite{BorodinOlshanski, Strahov}).
The terminology ``Jack measure'' is derived from Jack polynomials (\cite[VI.10]{Mac}).

When $\alpha=1$, we have $c_{\lambda}(1)= c'_{\lambda}(1)=H_{\lambda}$, 
where 
$$
H_{\lambda}=
\prod_{(i,j) \in \lambda} ((\lambda_i-j)+(\lambda_j'-i)+1)
$$
is the hook-length product.
By the well-known hook formula (see e.g. \cite[Theorem 3.10.2]{Sagan})
\begin{equation} \label{eq:hookformula}
f^\lambda=n! / H_\lambda,
\end{equation}
the measure $\bP^{\mr{Jack},1}_{n}$ is just the ordinary Plancherel measure $\bP^{\mr{Plan}}_n$ 
defined in \eqref{eq:Plancherel}.
The measure $\bP^{\mr{Jack},2}_n$ is the Plancherel measure
associated with the Gelfand pair $(\mf{S}_{2n},K_n)$, 
where $K_n (=\mf{S}_2 \wr \mf{S}_n)$ 
is the hyperoctahedral group in $\mf{S}_{2n}$, see \cite[\S 4.4]{FulmanCharacter}.
From the equality $c'_{\lambda}(\alpha)= \alpha^{|\lambda|} c_{\lambda'}(\alpha^{-1})$,
we have the duality
$$
\bP^{\mr{Jack},\alpha}_{n}(\lambda)=\bP^{\mr{Jack},\alpha^{-1}}_{n}(\lambda'),
$$
for any $\lambda \in \mcal{P}_n$ and $\alpha >0$.

Denote by $\mcal{P}_n(d)$ 
the set of partitions $\lambda$ in $\mcal{P}_n$
of length $\le d$.
We consider the restricted Jack measure with parameter $\alpha$ on $\mcal{P}_n(d)$:
\begin{equation}
\bP^{\mr{Jack},\alpha}_{n,d}(\lambda)= 
\frac{1}{C_{n,d}(\alpha)}
\frac{1}{c_{\lambda}(\alpha)c'_{\lambda}(\alpha)},
\qquad \lambda \in \mcal{P}_n(d),
\end{equation}
where 
\begin{equation} \label{eq:ConstantC}
C_{n,d}(\alpha)= \sum_{\mu \in \mcal{P}_n(d)} 
\frac{1}{c_{\mu}(\alpha)c'_{\mu}(\alpha)}.
\end{equation}
By the definition \eqref{eq:defJackMeasure} of the Jack measure,
 $C_{n,d}(\alpha)= (\alpha^n n!)^{-1}$ if $d \ge n$.

\subsection{Traceless Gaussian matrix ensembles}

Let
$$
\mf{H}_{d} =
\{(x_1,\dots,x_d) \in \bR^d \ | \ x_1 \ge \cdots \ge x_d, \  x_1+ \cdots +x_d=0\}.
$$
and let $\beta$ be a positive real number.
We equip the set $\mf{H}_d$
with  the probability density function 
\begin{equation} \label{eq:tlGBEdensity}
\frac{1}{Z_d(\beta)} e^{-\frac{\beta}{2} \sum_{j=1}^d x_j^2} 
\prod_{1 \le j <k \le d} (x_j-x_k)^\beta,
\end{equation}
where the normalization constant $Z_d(\beta)$ is defined by 
\begin{equation} \label{eq:ConstantZ}
Z_{d}(\beta)=\int_{\mf{H}_{d}} e^{-\frac{\beta}{2} \sum_{j=1}^d x_j^2} 
\prod_{1 \le j <k \le d} (x_j-x_k)^\beta d x_1 \cdots d x_{d-1}.
\end{equation}
Here the integral runs over $(x_1,\dots,x_{d-1}) \in \bR^{d-1}$ such that
$(x_1,\dots,x_d) \in \mf{H}_d$ with $x_d:=-(x_1 + \cdots+x_{d-1})$.
The explicit expression of $Z_d(\beta)$ is obtained in \cite{Regev} but we do not need it here.
We call the set $\mf{H}_{d}$ with probability density \eqref{eq:tlGBEdensity}
the {\it traceless Gaussian $\beta$-ensemble} ($\GBE{\beta}_0$).

If $\beta=1,2$, or $4$, the $\GBE{\beta}_0$ 
gives the distribution of the eigenvalues 
of a traceless Gaussian random matrix $X$ as follows
(see 
\cite{Forrester, Mehta}).

Let $\beta=1$.
We equip the space of $d \times d$ symmetric real matrices $X$ such that $\tr X=0$
with the probability density function
proportional to $e^{-\frac{1}{2} \tr (X^2)}$.
Then we call the random matrix $X$ a traceless Gaussian orthogonal ensemble 
($\GOE_0$) random matrix.
Let $\beta=2$.
Then we consider the space of $d \times d$ Hermitian complex matrices $X$ such that $\tr X=0$ 
with the probability density function
proportional to $e^{- \tr (X^2)}$.
We call $X$ a traceless Gaussian unitary ensemble ($\GUE_0$) random matrix.
Let $\beta=4$.
Then we consider the space of $d \times d$ Hermitian quaternion matrices $X$ such that $\tr X=0$
with the probability density function proportional to $e^{- \tr (X^2)}$.
We call $X$ a traceless Gaussian symplectic ensemble ($\GSE_0$) random matrix.

The $\GOE_0$, $\GUE_0$, and $\GSE_0$ matrices are 
the restriction of the ordinary GOE, GUE, GSE matrices   
to matrices whose trace is zero.
From the well-known fact in random matrix theorey,
the probability density function of eigenvalues of $X$ is given by 
\eqref{eq:tlGBEdensity} with $\beta=1$ ($\GOE_0$), $\beta=2$ ($\GUE_0$), or 
$\beta=4$ ($\GSE_0$).
We note that 
for general $\beta>0$, Dumitriu and Edelman \cite{DE} give tridiagonal matrix models
for Gaussian $\beta$-ensembles.

\subsection{Main theorem}

Let $(x_{1}^{\GBE{\beta}_0}, x_{2}^{\GBE{\beta}_0}, \dots,x_{d}^{\GBE{\beta}_0})$ be 
a sequence of random variables according to the $\GBE{\beta}_0$.
Equivalently, the joint probability density function for $(x_i^{\GBE{\beta}_0})_{1 \le i \le d}$
is given by \eqref{eq:tlGBEdensity}.
Our main result is as follows.

\begin{thm} \label{maintheorem}
Let $\alpha$ be any positive real number and put $\beta=2/\alpha$.
Let $\lambda^{(n)}=(\lambda_1^{(n)},\dots,\lambda_d^{(n)})$ be 
a random partition in $\mcal{P}_n(d)$  chosen with probability 
$\bP^{\mr{Jack},\alpha}_{n,d}(\lambda^{(n)})$.
Then, as $n \to \infty$,
the random variables 
$$
\( \sqrt{\frac{\alpha d}{n}}\(\lambda_{i}^{(n)} -\frac{n}{d} \)\)_{1 \le i \le d}
$$
converge to $\( x_{i}^{\GBE{\beta}_0}\)_{1 \le i \le d}$
in joint distribution.
\end{thm}

The case with $\alpha=1$ (and so $\beta=2$) of Theorem \ref{maintheorem} is proved in \cite{Sniady}.
(We remark that the definition of the density of a $\GUE_0$ matrix in \cite{Sniady}
is slightly different from us.)

We give the proof of Theorem \ref{maintheorem} in Section \ref{sec:Proof}.

%
%%%%%%%%%%%%%%%%%%%%%%%%%%%%%%%%%%%%%%%%%%%%%%%%%%%%%%%%%%%%%%%%%%%%%%%%%%%%%%%%%%%%%%%%%%%%%%%%
\section{Jack measures with $\alpha=2$ or $\frac{1}{2}$ and RSK correspondences} \label{sec:RSK} 
%%%%%%%%%%%%%%%%%%%%%%%%%%%%%%%%%%%%%%%%%%%%%%%%%%%%%%%%%%%%%%%%%%%%%%%%%%%%%%%%%%%%%%%%%%%%%%%%
%

In this section, we deal with Jack measures with parameter $\alpha=2$ and $\alpha=\frac{1}{2}$.
Our goal is to obtain a limit theorem for 
a random involutive permutation as a corollary of Theorem \ref{maintheorem}.

\begin{lem} \label{lem:Hook}
For each $\lambda \in \mcal{P}_n$ we have
\begin{equation}
c_{\lambda}(2) c'_{\lambda}(2)= H_{2\lambda}, \qquad
c_{\lambda}(1/2) c'_{\lambda}(1/2)= 2^{-2n} H_{\lambda \cup \lambda},
\end{equation}
where $2\lambda=(2\lambda_1,2\lambda_2,\dots)$
 and $\lambda \cup \lambda= (\lambda_1,\lambda_1,\lambda_2,\lambda_2,\dots)$.
\end{lem}

\begin{proof}
Put $\mu=2\lambda$.
Then, since $\mu_i=2\lambda_i$ and $\mu'_{2j-1}=\mu'_{2j} = \lambda'_j$ for any $i,j \ge 1$, 
we have
\begin{align*}
H_{\mu}=& \prod_{\begin{subarray}{c} (i,j) \in \mu, \\ j: \mr{odd} \end{subarray}}
(\mu_i-j+\mu'_j-i+1) \times
\prod_{\begin{subarray}{c} (i,j) \in \mu, \\ j: \mr{even} \end{subarray}}
(\mu_i-j+\mu'_j-i+1) \\
=& \prod_{(i,j) \in \lambda} (2\lambda_i-(2j-1) +\lambda_j'-i+1) \times  
\prod_{(i,j) \in \lambda} (2\lambda_i-2j +\lambda_j'-i+1) \\
=&c'_{\lambda}(2)c_{\lambda}(2).
\end{align*}
Applying 
the equality $c'_{\lambda}(\alpha)= \alpha^n c_{\lambda'}(\alpha^{-1})$ to 
the above identity,
we see that
$2^{2n} c_{\lambda}(1/2)c'_{\lambda}(1/2)=H_{2\lambda'} =H_{(\lambda \cup \lambda)'}
=H_{\lambda \cup \lambda}$.
\end{proof}

By this lemma, the Jack measures with parameter $\alpha=2$ and $\frac{1}{2}$ are
expressed as follows.
$$
\bP^{\mr{Jack},2}_{n}(\lambda) = \frac{f^{2\lambda}}{(2n-1)!!},
\qquad 
\bP^{\mr{Jack},\frac{1}{2}}_{n}(\lambda) = \frac{f^{\lambda \cup \lambda}}{(2n-1)!!}.
$$

Recall the Robinson-Schensted-Knuth(RSK) correspondence (see e.g. \cite[Chapter 3]{Sagan}).
There exists a one-to-one correspondence 
between elements in $\mf{S}_N$ and ordered pairs of standard Young tableaux 
of same shape whose size is $N$ (\cite[Theorem 3.1.1]{Sagan}).
Let $\sigma \in \mf{S}_N$ correspond to the ordered pair $(P,Q)$ of 
standard Young tableaux of shape $\mu \in \mcal{P}_N$.
Then, the length $L^{\mr{in}}(\sigma)$ of the longest increasing subsequence in 
$(\sigma(1),\dots,\sigma(N))$ 
is equal to $\mu_1$.
Similarly, 
the length $L^{\mr{de}}(\sigma)$ of the longest decreasing subsequence in $\sigma$ 
is equal to $\mu_1'$ (\cite[Theorem 3.3.2]{Sagan}).
Furthermore, the permutation $\sigma^{-1}$ corresponds to the pair $(Q,P)$ 
(\cite[Theorem 3.6.6]{Sagan}).
In particular, there exists a one-to-one correspondence between
involutions $\sigma$ (i.e. $\sigma=\sigma^{-1})$ in $\mf{S}_N$ and
standard Young tableaux of size $N$.

Let $\sigma$ be an involution with $k$ fixed points.
Then the standard Young tableau corresponding to $\sigma$ has exactly $k$ columns of odd length
(\cite[Exercises 3.12.7(b)]{Sagan}).
Therefore, the number of fixed-point-free involutions $\sigma$ in $\mf{S}_{2n}$ such that 
$L^{\mr{in}}(\sigma) \le a$ and 
$L^{\mr{de}}(\sigma) \le 2b$
is equal to 
$$
\sum_{\begin{subarray}{c} \mu \in \mcal{P}_{2n} \\ \mu':\mathrm{even} \\
\mu_1 \le a, \  \mu_1' \le 2b
\end{subarray}} f^{\mu}
= \sum_{\begin{subarray}{c} \lambda \in \mcal{P}_n \\ 
\lambda_1 \le a,\ \ell(\lambda) \le b \end{subarray}} f^{\lambda \cup \lambda}
=\sum_{\begin{subarray}{c} \lambda \in \mcal{P}_n \\ 
\lambda_1 \le b,\ \ell(\lambda) \le a \end{subarray}} f^{2\lambda},
$$
where the first sum runs over partitions $\mu$ in $\mcal{P}_{2n}$ 
whose conjugate partition $\mu'$ is even, (i.e. all $\mu'_j$ are even,) 
satisfying $\mu_1 \le a$ and $\mu_1' \le 2b$.

Note that the values $C_{n,d}(1/2)$ and $C_{n,d}(2)$ are expressed by a matrix integral.
Using Rains' result \cite{Rains}, we have
$$
C_{n,d}(1/2)=\frac{2^{2n}}{(2n)!} \sum_{\lambda \in \mcal{P}_n(d)} f^{\lambda \cup \lambda}
= \frac{2^{2n}}{(2n)!} \int_{Sp(2d)} \tr (S)^{2n} d S,
$$
where the integral runs over the symplectic group with its normalized Haar measure.
Similarly, 
$$
C_{n,d}(2)=\frac{1}{(2n)!} \sum_{\lambda \in \mcal{P}_n(d)} f^{2\lambda}
= \frac{1}{(2n)!} \int_{O(d)} \tr (O)^{2n} d O,
$$
where the integral runs over the orthogonal group with its normalized Haar measure.

Let $\mf{S}^0_{2n}$ be the subset in $\mf{S}_{2n}$ of fixed-point-free involutions.
Equivalently, 
$$
\mf{S}_{2n}^0=\{\sigma \in \mf{S}_{2n} \ | \ \text{The cycle-type of $\sigma$ 
is $(2^n)$} \}.
$$
We pick $\sigma \in \mf{S}_{2n}^0$ at random according to the uniformly distributed probability,
i.e. the probability of all $\sigma \in \mf{S}_{2n}^0$ are equal.

\begin{lem}
\begin{enumerate}
\item
The distribution function $\bP^{\mr{Jack},1/2}_{n,d}(\lambda_1 \le h)$
of the random variable $\lambda_1$ with respect to 
$\bP^{\mr{Jack},1/2}_{n,d}(\lambda)$ 
is equal to the ratio
$$
\frac{\#\{\sigma \in \mf{S}_{2n}^0 \ | \ L^{\mr{de}}(\sigma) \le 2d \quad 
\text{and} \quad L^{\mr{in}}(\sigma) \le h  \}}
{\#\{\sigma \in \mf{S}_{2n}^0 \ | \ L^{\mr{de}}(\sigma) \le 2d \}},
$$
which is the distribution function of $L^{\mr{in}}$ for a random involution 
$\sigma \in \mf{S}_{2n}^0$
such that $L^{\mr{de}}(\sigma) \le 2d$.
\item
The distribution function $\bP^{\mr{Jack},2}_{n,d}(\lambda_1 \le h)$
of the random variable $\lambda_1$ with respect to 
$\bP^{\mr{Jack},2}_{n,d}(\lambda)$ 
is equal to the ratio
$$
\frac{\#\{\sigma \in \mf{S}_{2n}^0 \ | \ L^{\mr{in}}(\sigma) \le d \quad 
\text{and} \quad L^{\mr{de}}(\sigma) \le 2h  \}}
{\#\{\sigma \in \mf{S}_{2n}^0 \ | \ L^{\mr{in}}(\sigma) \le d \}},
$$
which is the distribution function of $\frac{1}{2}L^{\mr{de}}$ for a random involution 
$\sigma \in \mf{S}_{2n}^0$
such that $L^{\mr{in}}(\sigma) \le d$.
\end{enumerate}
\end{lem}

By the above lemma and Theorem \ref{maintheorem}, we obtain the following corollary.

\begin{cor}
\begin{enumerate}
\item (The $\alpha=1/2$ case)
Let $\sigma \in \mf{S}_{2n}^0$ be a random fixed-point-free involution
with the longest decreasing subsequence of length at most $2d$.
Then, as $n \to \infty$,
the distribution of $\sqrt{\frac{d}{2n}}\(L^{\mr{in}}(\sigma) -\frac{n}{d}\)$
converges to the distribution for the largest eigenvalue of a $\GSE_0$ random matrix of size $d$.
\item (The $\alpha=2$ case)
Let $\sigma \in \mf{S}_{2n}^0$ be a random fixed-point-free involution
with the longest increasing subsequence of length at most $d$.
Then, as $n \to \infty$,
the distribution of $\sqrt{\frac{2d}{n}}\(\frac{L^{\mr{de}}(\sigma)}{2} -\frac{n}{d}\)$
converges to the distribution for the largest eigenvalue of the $\GOE_0$ random matrix of size $d$.
\end{enumerate}
\end{cor}

The $\alpha=1$ version of this corollary appears in \cite[Corollary 4]{Sniady}.

%
%%%%%%%%%%%%%%%%%%%%%%%%%%%%%%%%%%%%%%%%%%%%%%%%%%%%%%%%%%%%%%%%%%%%%%%%%%%%%%%%%%%%%%%%%%%%%%%%
\section{Proof of Theorem \ref{maintheorem}} \label{sec:Proof}
%%%%%%%%%%%%%%%%%%%%%%%%%%%%%%%%%%%%%%%%%%%%%%%%%%%%%%%%%%%%%%%%%%%%%%%%%%%%%%%%%%%%%%%%%%%%%%%%
%

\subsection{Step 1} 
\label{subsec:BOexpression}

The following explicit formula for $c_\lambda(\alpha)$ and $c'_{\lambda}(\alpha)$
appears in the proof of Lemma 3.5 in \cite{BorodinOlshanski}.

\begin{lem} \label{lem:ccExplicit}
For any $\alpha>0$ and $\lambda \in \mcal{P}_n(d)$, 
\begin{align*}
c_\lambda(\alpha)=& \alpha^n \prod_{1 \le i<j \le d} 
\frac{\Gamma(\lambda_i-\lambda_j +(j-i)/\alpha)}{\Gamma(\lambda_i-\lambda_j +(j-i+1)/\alpha)}
\cdot \prod_{i=1}^d \frac{\Gamma(\lambda_i+(d-i+1)/\alpha)}{\Gamma(1/\alpha)}, \\
c'_\lambda(\alpha)=& \alpha^n \prod_{1 \le i<j \le d} 
\frac{\Gamma(\lambda_i-\lambda_j +(j-i-1)/\alpha+1)}{\Gamma(\lambda_i-\lambda_j +(j-i)/\alpha+1)}
\cdot \prod_{i=1}^d \Gamma(\lambda_i+(d-i)/\alpha+1).
\end{align*}
\end{lem}

\begin{proof}
For each $i \ge 1$, let  $m_i'=m_i(\lambda')$ be the multiplicity of $i$ 
in $\lambda'=(\lambda_1',\lambda_2',\dots)$. 
Then one observes
$$
\prod_{i=1}^{r} \prod_{j: \lambda_j'=r}  (\lambda_i-j +(\lambda_j'-i+1)/\alpha) 
= \prod_{i=1}^r 
\prod_{p=1}^{m_r'}  (m_i'+m_{i+1}' + \cdots +m_{r-1}' +p-1 +(r-i+1)/\alpha )
$$
for each $1 \le r \le d$.
Since $m_i'=\lambda_i-\lambda_{i+1}$, we have 
\begin{align*}
c_{\lambda}(\alpha)=& 
\alpha^n \prod_{(i,j) \in \lambda}  (\lambda_i-j +(\lambda_j'-i+1)/\alpha) \\
=&\alpha^n \prod_{r = 1}^d
\prod_{i=1}^r 
\prod_{p=1}^{\lambda_r-\lambda_{r+1} } (\lambda_i-\lambda_{r} +p-1 +(r-i+1)/\alpha ) \\
=& \alpha^n \prod_{i=1}^d \prod_{r=i}^d 
\frac{((r-i+1)/\alpha)_{\lambda_i-\lambda_{r+1} } }{((r-i+1)/\alpha)_{\lambda_i-\lambda_{r}}}.
\end{align*}
Here $(a)_k=\Gamma(a+k)/\Gamma(a)$ is the Pochhammer symbol.
We moreover see  that
\begin{align*}
\alpha^{-n} c_{\lambda}(\alpha)=&  
\prod_{i=1}^d \left[
\frac{(1/\alpha)_{\lambda_i-\lambda_{i+1}}}{1} 
\frac{(2/\alpha)_{\lambda_i-\lambda_{i+2}}}{(2/\alpha)_{\lambda_i-\lambda_{i+1}}}
\cdots \frac{((d-i+1)/\alpha)_{\lambda_i-\lambda_{d+1}}}{((d-i+1)/\alpha)_{\lambda_i-\lambda_{d}}}
\right] \\
=&\prod_{1 \le i<j \le d} 
\frac{((j-i)/\alpha)_{\lambda_i-\lambda_{j}}}{((j-i+1)/\alpha)_{\lambda_i-\lambda_{j}}}
\cdot \prod_{i=1}^d ((d-i+1)/\alpha)_{\lambda_i}.
\end{align*}
Now the first product equals
\begin{align*}
&
\prod_{1 \le i<j \le d} \frac{\Gamma(\lambda_i-\lambda_j +(j-i)/\alpha)}
{\Gamma((j-i)/\alpha)}
\frac{\Gamma((j-i+1)/\alpha)}{\Gamma(\lambda_i-\lambda_j +(j-i+1)/\alpha)} \\
=& \prod_{1 \le i<j \le d} \frac{\Gamma(\lambda_i-\lambda_j +(j-i)/\alpha)}
{\Gamma(\lambda_i-\lambda_j +(j-i+1)/\alpha)}
\cdot \prod_{i=1}^{d-1} \frac{\Gamma((d-i+1)/\alpha)}{\Gamma(1/\alpha)},
\end{align*}
and the second product equals
$$
\prod_{i=1}^d \frac{\Gamma(\lambda_i +(d-i+1)/\alpha)}{\Gamma((d-i+1)/\alpha)}.
$$
Thus we obtain the desired expression for $c_\lambda(\alpha)$.
Similarly for $c'_\lambda(\alpha)$.
\end{proof}

\subsection{Step 2} \label{subsec:Sniady}

The discussion in this subsection is a slight generalization of the one in \cite{Sniady}.

We put
$$
\xi_r^{(n)} = \frac{r-\frac{n}{d}}{\sqrt{\frac{n}{d}}}
$$
for each $r \in \bZ$.
For any positive real number $\theta >0$, 
we define the function $\phi_{n;\theta}: \bR \to \bR$ which is constant on the interval
of the form $I_r^{(n)}=[\xi^{(n)}_r, \xi^{(n)}_{r+1})$ for each integer $r$, and 
such that
$$
\phi_{n;\theta}(\xi_r^{(n)})  =
\begin{cases}
\frac{1}{  {_1}F_1(1;\theta;\frac{n}{d})}
\frac{\(\frac{n}{d}\)^{r+\frac{1}{2}}}{(\theta)_r} &
\text{if $r$ is non-negative}, \\
0 & \text{if $r$ is negative}. 
\end{cases}
$$
Here ${_1}F_{1}(a;b;x)= \sum_{r=0}^\infty \frac{(a)_r}{(b)_r} \frac{x^r}{r!}$
is the hypergeometric function of type $(1,1)$.
The following asymptotics follows from \cite[Corollary 4.2.3]{AAR}:
\begin{equation} \label{eq:AsymptoticHypergeometric}
{_1}F_1\(1;\theta;\frac{n}{d}\) \sim
\frac{e^{\frac{n}{d}} \Gamma(\theta)}{\(\frac{n}{d}\)^{\theta-1}}
\qquad \text{as $n \to \infty$}. 
\end{equation}

The function $\phi_{n;\theta}$ is a probability density function on $\bR$.
Indeed, since $\bR= \bigsqcup_{r \in \bZ} I_r^{(n)}$ and 
since the volume of each $I_r^{(n)}$ is $\xi_{r+1}^{(n)}-\xi_{r}^{(n)}=\sqrt{\frac{d}{n}}$,
we have
$$
\int_{\bR} \phi_{n;\theta}(y) d y = \sum_{r=0}^ \infty \sqrt{\frac{d}{n}} 
\phi_{n;\theta} (\xi_r^{(n)})
= \frac{1}{{_1}F_1(1;\theta;\frac{n}{d})}
\sum_{r=0}^\infty 
\frac{\(\frac{n}{d}\)^r}{(\theta)_r} =1.
$$
We often need the equation
\begin{equation} \label{eq:factorialPhi}
\frac{1}{\Gamma(r+\theta)} = 
\frac{{_1}F_1(1;\theta;\frac{n}{d})}
{\(\frac{n}{d}\)^{r+\frac{1}{2}} \Gamma(\theta)}
\phi_{n;\theta} ( \xi_r^{(n)} ) 
\sim
 \frac{e^{\frac{n}{d}}}{\(\frac{n}{d}\)^{r+\theta-\frac{1}{2}} } \phi_{n;\theta}(\xi^{(n)}_r),
\end{equation}
as $n \to \infty$,
for any fixed $\theta>0$ and a non-negative integer $r$.
Here we have used \eqref{eq:AsymptoticHypergeometric}.

The following lemma generalizes  \cite[Lemma 5]{Sniady}  slightly.

\begin{lem}
For any $\theta>0$ and $y \in \bR$, we have
\begin{equation} \label{eq:LimitPhi}
\lim_{n \to \infty} \phi_{n;\theta}(y) = \frac{1}{\sqrt{2\pi}} e^{-\frac{y^2}{2}}.
\end{equation}
Furthermore, there exists a constant $C=C_\theta$ such that
\begin{equation} \label{eq:IneqConstant}
\phi_{n;\theta}(y)  < C e^{-|y|}
\end{equation}
holds true for all $n$ and $y$.
\end{lem}

\begin{proof}
Fix $y \in \bR$.
Let $c= \frac{n}{d}$ and let $r=r(y,c)$ be an integer such that $y \in I_r^{(n)}$,
i.e., $r=\lfloor c+y\sqrt{c} \rfloor$.
We may suppose that $r$ is positive
because $r$ is large when $n$ is large. 
By \eqref{eq:factorialPhi} and the asymptotics 
\begin{equation} \label{eq:GammaAsymptotic}
\Gamma(r+\theta) \sim \Gamma(r) r^\theta \qquad \text{for $\theta$ fixed and as $r \to \infty$}, 
\end{equation}
we see that
$$
\phi_{n;\theta}(y) 
\sim c^{\theta-1}
\frac{e^{-c} c^{r+\frac{1}{2}}}{\Gamma(\theta+r)}
\sim \(\frac{c}{r}\)^{\theta-1}
\frac{e^{-c} c^{r+\frac{1}{2}} }{r!}
=\(\frac{c}{r}\)^{\theta-1} \phi_{n;1}(y) \sim \phi_{n;1}(y)
$$
as $n \to \infty$ (so $c \to \infty$).
Therefore we may assume $\theta=1$ in order to prove \eqref{eq:LimitPhi}.
Using Stirling's formula $\log r! =\(r+\frac{1}{2}\) \log r - r + \frac{\log 2\pi}{2} +O(r^{-1})$,
we have
\begin{align*}
& \log  \phi_{n;1}(y) =\log  \phi_{n;1}(\xi_{r}^{(n)}) = 
\(r+\frac{1}{2} \) \log c -c  - \log r! \\
=& -\(c+\xi_r^{(n)} 
\sqrt{c}+\frac{1}{2} \) \log\(1+ \frac{\xi_r^{(n)}}{\sqrt{c}}\) +
\xi_r^{(n)}\sqrt{c} -\frac{\log 2\pi}{2}
+O(c^{-1}) \\
=& -\(c+\xi_r^{(n)}\sqrt{c}+\frac{1}{2} \) 
\( \frac{\xi_r^{(n)}}{\sqrt{c}} - \frac{(\xi_r^{(n)})^2}{2c} +O(c^{-\frac{3}{2}})\) 
+\xi_r^{(n)}\sqrt{c} -\frac{\log 2\pi}{2} +O(c^{-1}) \\
=& -\frac{(\xi_r^{(n)})^2}{2} -\frac{\log 2\pi}{2} +O(c^{-\frac{1}{2}})
\end{align*}
as $c \to \infty$.
Since $y =\xi_{r}^{(n)}+O(c^{-\frac{1}{2}})$, we obtain \eqref{eq:LimitPhi}.

In order to prove \eqref{eq:IneqConstant},
we consider the function
$$
g_\theta(y,a)= -a \log \(1+\frac{y}{a}+\frac{\theta-1}{a^2}  \)
$$
for $(y,a) \in \bR \times \bR_{>0}$.
Then, for each positive integer $r$,
$$
\frac{\log \phi_{n;\theta} (\xi_{r}^{(n)})- \log \phi_{n;\theta}(\xi_{r-1}^{(n)})}
{\xi_{r}^{(n)}-\xi_{r-1}^{(n)}}
=-  \sqrt{\frac{n}{d}} \log 
\(1+\frac{\xi_{r}^{(n)}}{\sqrt{\frac{n}{d}}}+\frac{\theta-1}{\frac{n}{d}}  \) 
= g_\theta \(\xi_r^{(n)},\sqrt{\frac{n}{d}}\).
$$
It is easy to see that
$g_\theta(y,a) >1 \ \Longleftrightarrow \ y < a (e^{-1/a}-1) -\frac{\theta-1}{a}$.
Take a negative number $D_1$ such that $D_1 <-1+\min \{0, -(\theta-1)\sqrt{d}\}$. 
Since $a (e^{-1/a}-1) >-1$ for all $a >0$, 
if $\xi_r^{(n)} <D_1$, 
then we have $\xi_r^{(n)} <  \sqrt{\frac{n}{d}} (e^{-\sqrt{\frac{d}{n}}}-1) 
-\frac{\theta-1}{\sqrt{\frac{n}{d}}}$,
which is equivalent to $g_\theta \(\xi_r^{(n)},\sqrt{\frac{n}{d}}\) >1$.
Therefore
\begin{equation} \label{eq:Inequality1}
\phi_{n;\theta}(\xi_{r-1}^{(n)}) e^{-\xi_{r-1}^{(n)}} < 
 \phi_{n;\theta}(\xi_{r}^{(n)}) e^{-\xi_{r}^{(n)}}.
\end{equation}

Similarly, we see that $g_\theta(y,a) <-1 \Longleftrightarrow \
  y > a(e^{1/a}-1)- \frac{\theta-1}{a}$.
Since the function $a(e^{1/a}-1)$ in $a$ is monotonically decreasing on $(0,\infty)$
and $\lim_{a \to +0} a(e^{1/a}-1)=+\infty$,
we can take a large positive constant $D_2'$ such that $\sqrt{\frac{n}{d}} (e^{\sqrt{\frac{d}{n}}}-1)<D_2'$
for all $n$.
Take a positive number $D_2$ such that $D_2>D_2' + \max\{0,-(\theta-1)\sqrt{d}\}$.
Then, if $\xi_{r}^{(n)} >D_2$, 
we have $\xi_{r}^{(n)}> \sqrt{\frac{n}{d}}(e^{\sqrt{\frac{d}{n}}}-1)-\frac{\theta-1}{\sqrt{\frac{n}{d}}}$ 
for any $n$,
which is equivalent to $ g_\theta\(\xi_{r}^{(n)},\sqrt{\frac{n}{d}}\) <-1$.
Therefore
\begin{equation} \label{eq:Inequality2}
\phi_{n;\theta}(\xi_{r}^{(n)}) e^{\xi_{r}^{(n)}} < 
\phi_{n;\theta}(\xi_{r-1}^{(n)}) e^{\xi_{r-1}^{(n)}}.
\end{equation}

The equation \eqref{eq:Inequality1} implies that
there exists a constant $C_1$ such that 
$\phi_{n;\theta}(y)e^{|y|} < C_1 $ for $y$ sufficiently smaller than $D_1$.
Similarly, the equation \eqref{eq:Inequality2} implies that
there exists a constant $C_2$ such that 
$\phi_{n;\theta}(y)  e^{|y|}< C_2$ for $y$ sufficiently bigger than $D_2$.
When $y$ belongs to a neighborhood of the interval $[ D_1,D_2 ] $,
the inequality \eqref{eq:IneqConstant} holds from the convergence \eqref{eq:LimitPhi}. 
Thus, there exists a constant $C$ such that $\phi_{n;\theta}(y) <C e^{-|y|}$ for all $y$.
\end{proof}

\subsection{Step 3}

By Lemma \ref{lem:ccExplicit} we have
$$
\frac{1}{c_\lambda(\alpha)c'_\lambda(\alpha)} 
= \alpha^{-2n} \Gamma(1/\alpha)^d 
\prod_{i=1}^4 F_i(\lambda_1,\dots,\lambda_d)
$$
where the functions $F_i$ are defined by
\begin{align*}
F_1(r_1,\dots,r_d)=& \prod_{1 \le i<j \le d} (r_i-r_j+(j-i)/\alpha),  \\
F_2(r_1,\dots,r_d)=& \prod_{1 \le i<j \le d} 
\frac{\Gamma(r_i-r_j+(j-i+1)/\alpha)}{\Gamma(r_i-r_j+(j-i-1)/\alpha+1)}, \\
F_3(r_1,\dots,r_d) =& \prod_{i=1}^d \frac{1}{\Gamma(r_i+(d-i+1)/\alpha)}, \\
F_4(r_1,\dots,r_d) =& \prod_{i=1}^d \frac{1}{\Gamma(r_i+(d-i)/\alpha+1)},
\end{align*}
for real numbers $r_1 \ge \cdots \ge r_d \ge 0$.

\begin{lem} \label{lem:DenstiyAsymptotic1}
Let $(y_1,\dots,y_d) \in \mf{H}_d$ be real numbers  such that $y_1>y_2>\cdots>y_d$,
and let $r_i= \frac{n}{d}+y_i \sqrt{\frac{n}{d}}$ for $1 \le i \le d$.
Then as $n \to \infty$
\begin{align*}
F_1(r_1,\dots,r_d) \sim & \( \frac{n}{d} \)^{\frac{d(d-1)}{4}} \prod_{1 \le i<j \le d}
(y_i-y_j), \\
F_2(r_1,\dots,r_d) \sim & \( \frac{n}{d} \)^{\frac{d(d-1)}{4} (\frac{2}{\alpha}-1)} 
\prod_{1 \le i<j \le d}
(y_i-y_j)^{\frac{2}{\alpha}-1}, \\
F_3(r_1,\dots,r_d) \sim & \frac{e^{n}}{
\( \frac{n}{d} \)^{n+\frac{d(d+1)}{2\alpha} -\frac{d}{2}} }
\prod_{i=1}^d \frac{e^{-\frac{y_i^2}{2}}}{\sqrt{2\pi}},\\ 
F_4(r_1,\dots,r_d) \sim & \frac{e^{n}}{
\( \frac{n}{d} \)^{n+\frac{d(d-1)}{2\alpha} +\frac{d}{2}}}
\prod_{i=1}^d \frac{e^{-\frac{y_i^2}{2}}}{\sqrt{2\pi}}.
\end{align*}
Moreover, there exists a function $P$ in $d$ variables such that
\begin{equation} \label{eq:Dominated}
\frac{\(\frac{n}{d}\)^{2n+\frac{d+d^2}{2\alpha} }}{e^{2n}}
\prod_{i=1}^4 F_i(r_1,\dots, r_d) < P(y_1,\dots, y_d) \prod_{i=1}^d e^{-2|y_i|}
\end{equation}
for all $y_1,\dots,y_d$ and $n$, and such that
\begin{equation} \label{eq:orderP}
P(y_1,\dots,y_d)=O(|y_1|^{k_1} \cdots |y_d|^{k_d}) \qquad 
\text{as $|y_1|,\dots,|y_d| \to \infty$}
\end{equation}
with some positive real numbers $k_1,\dots, k_d$.
\end{lem}

\begin{proof}
It is immediate to see that 
\begin{equation} \label{Asym:F1}
F_1(r_1,\dots,r_d) = \(\frac{n}{d}\)^{\frac{d(d-1)}{4}}
\prod_{1 \le i< j \le d} \(y_i-y_j + \sqrt{\frac{d}{n}}(j-i)/\alpha \)
\end{equation}
so that the desired asymptotics for $F_1$ follows.
The asymptotics for $F_2$ also follows by \eqref{eq:GammaAsymptotic}.

Using \eqref{eq:factorialPhi}, we have
$$
F_3(r_1,\dots, r_d) \sim \prod_{i=1}^d \frac{e^{\frac{n}{d}}}
{\(\frac{n}{d}\)^{r_i+(d-i+1)/\alpha-1/2}}
\phi_{n;(d-i+1)/\alpha} (y_i)
= \frac{e^n}{\(\frac{n}{d}\)^{n+\frac{d(d+1)}{2\alpha} -\frac{d}{2}}}
\prod_{i=1}^d \phi_{n;(d-i+1)/\alpha} (y_i).
$$
The desired asymptotics for $F_3$  follows from \eqref{eq:LimitPhi}.
Similarly for $F_4$.

Observe that there exist positive constants $c_{ij}$ and  $d_{ij}$ such that
$$
(r_i-r_j +(j-i)/\alpha) \frac{\Gamma(r_i-r_j+(j-i+1)/\alpha)}{\Gamma(r_i-r_j+(j-i-1)/\alpha+1)}
< c_{ij} (r_i-r_j)^{2/\alpha} +d_{ij}
$$
for any $r_1 \ge \cdots \ge r_d$.
This implies that 
there exists a function $P'$ in $d$ variables such that
\begin{equation} \label{eq:Dominate1}
F_1(r_1,\dots,r_d) F_2(r_1,\dots,r_d) < \(\frac{n}{d}\)^{\frac{d(d-1)}{2\alpha}}
P'(y_1,\dots,y_d)
\end{equation}
for all $n$ and satisfying the asymptotics given in \eqref{eq:orderP}.
On the other hand, by \eqref{eq:factorialPhi},
we have
$$
F_3(r_1,\dots,r_d) F_4(r_1,\dots,r_d) < 
C'' \frac{e^{2n}}{\(\frac{n}{d}\)^{2n+\frac{d^2}{\alpha}}} \prod_{i=1}^d 
\phi_{n;(d-i+1)/\alpha}(y_i) \phi_{n;(d-i)/\alpha+1}(y_i)
$$
with some constant $C''$.
Therefore \eqref{eq:Dominated} follows from \eqref{eq:IneqConstant}.
\end{proof}

We extend $\frac{1}{c_\lambda(\alpha)c'_\lambda(\alpha)}$ as follows.
For all real numbers $r_1,\dots,r_d \in \bR$ satisfying $r_1 +\cdots+r_d=n$, we put
$$
\widehat{c}^{(\alpha)}(r_1,\dots,r_d)= 
\begin{cases}
\alpha^{-2n} \Gamma(1/\alpha)^d \prod_{i=1}^4 F_i(r_1,\dots,r_d) &
\text{if $r_1 \ge \cdots \ge r_d \ge 0$},\\
0 & \text{otherwise}.
\end{cases}
$$
It follows by the above discussions that, 
for each $(y_1,\dots,y_d) \in \mf{H}_d$,
putting $r_i =\frac{n}{d}+y_i \sqrt{\frac{n}{d}} \ (1 \le i \le d)$,
\begin{equation} \label{eq:DensityAsymptotic} 
\lim_{n \to \infty}
\frac{\alpha^{2n} (2\pi)^d \(\frac{n}{d}\)^{2n+\frac{d+d^2}{2\alpha} }}{\Gamma(1/\alpha)^d e^{2n}}
\widehat{c}^{(\alpha)}(r_1,\dots,r_d)
= e^{-y_1^2 -\cdots -y_d^2}
\prod_{1 \le i<j \le d} (y_i - y_j)^{2/\alpha}.
\end{equation}
Also, there exists a function $P$ satisfying \eqref{eq:orderP} such that 
\begin{equation} \label{eq:DensityDominated} 
 \frac{\alpha^{2n} \(\frac{n}{d}\)^{2n+\frac{d+d^2}{2\alpha} }}{\Gamma(1/\alpha)^d e^{2n}}
\widehat{c}^{(\alpha)}(r_1,\dots,r_d)
< 
e^{-2(|y_1| +\cdots+ |y_d|)} 
P(y_1,\dots,y_d)
\end{equation}
for all $n$ and $(y_1,\dots,y_d) \in \mf{H}_d$.

\subsection{Step 4}

\begin{lem} \label{lem:SumLimit}
$$
\lim_{n \to \infty} C_{n,d}(\alpha) 
\frac{\alpha^{2n} (2\pi)^d \(\frac{n}{d}\)^{2n+\frac{d^2+d}{2\alpha}-\frac{d-1}{2}}}
{ \Gamma(1/\alpha)^d e^{2n}}
= \alpha^{-\frac{d(d-1)}{2\alpha}-\frac{d-1}{2}} Z_d(2/\alpha),
$$
where $C_{n,d}(\alpha)$ is given by \eqref{eq:ConstantC} and 
$Z_d(\beta)$ is given by \eqref{eq:ConstantZ}.
\end{lem}

\begin{proof}
By \eqref{eq:DensityDominated}, we see that 
\begin{align}
& 
\frac{\alpha^{2n} \(\frac{n}{d}\)^{2n+\frac{d^2+d}{2\alpha}-\frac{d-1}{2}}}
{\Gamma(1/\alpha)^d e^{2n}} C_{n,d}(\alpha) \notag \\
<& \sum_{
\begin{subarray}{c} \lambda_1 \ge \cdots \ge \lambda_{d-1} \ge \lambda_d \ge 0  \\
\lambda_d:=n-(\lambda_1 +\cdots+\lambda_{d-1}) \end{subarray}} 
\( \sqrt{\frac{d}{n}}\)^{d-1} 
e^{-2(|\xi_{\lambda_1}^{(n)}|+\cdots +|\xi_{\lambda_d}^{(n)}|)}
P(\xi_{\lambda_1}^{(n)},\dots, \xi_{\lambda_d}^{(n)})
\label{eq:RiemannSum} 
\end{align}
for all $n$.
Each $\xi_r^{(n)}$ is picked up from the interval $I_r^{(n)}=[\xi_r^{(n)},\xi_{r+1}^{(n)})$,
whose volume is $\sqrt{\frac{d}{n}}$.
Therefore we regard the sum on \eqref{eq:RiemannSum} as  a Riemann sum,
and the sum converges
to the integral 
$$
\int_{\mf{H}_d} e^{-2(|y_1|+\cdots+|y_d|)} P(y_1,\dots,y_d) d y_1 \cdots d y_{d-1},
$$
where the integral runs over $(y_1,\dots,y_{d-1}) \in \bR^{d-1}$ 
such that $(y_1,\dots,y_d) \in \mf{H}_d$  for $y_d:=-(y_1+ \cdots +y_{d-1})$.
We can apply the dominated convergence theorem:
by \eqref{eq:DensityAsymptotic}, 
\begin{align*}
& 
\frac{\alpha^{2n} (2\pi)^d \(\frac{n}{d}\)^{2n+\frac{d^2+d}{2\alpha}-\frac{d-1}{2}}}
{\Gamma(1/\alpha)^d e^{2n}} C_{n,d}(\alpha) \\
=& \sum_{
\begin{subarray}{c} \lambda_1 \ge \cdots \ge \lambda_{d-1} \ge \lambda_d \ge 0 \\
\lambda_d:=n-(\lambda_1 +\cdots+\lambda_{d-1}) \end{subarray}} 
\( \sqrt{\frac{d}{n}}\)^{d-1} 
\widehat{c}^{(\alpha)} \(\frac{n}{d}+\xi_{\lambda_1}^{(n)}\sqrt{\frac{n}{d}},
\dots, \frac{n}{d}+\xi_{\lambda_d}^{(n)}\sqrt{\frac{n}{d}}\)
\frac{\alpha^{2n} (2\pi)^d \(\frac{n}{d}\)^{2n+\frac{d^2+d}{2\alpha}}}
{\Gamma(1/\alpha)^d e^{2n}}
\\
&\xrightarrow{n \to \infty}
\int_{\begin{subarray}{c}
(y_1,\dots,y_d) \in \mf{H}_d  \\ y_d:=-(y_1+ \cdots+y_{d-1}) \end{subarray}}
e^{-(y_1^2+\cdots +y_d^2)} \prod_{1 \le i<j \le d} (y_i - y_j)^{2/\alpha} d y_1 \cdots d y_{d-1}.
\end{align*}
Changing variables as $y_j= \alpha^{-1/2} x_j$, we obtain the lemma. 
\end{proof}

Our goal is to prove the following equation:
for any $1 \le k \le d$ and any $h_1,\dots, h_k \in \bR$, 
\begin{align*}
& \lim_{n \to \infty} \frac{1}{C_{n,d}(\alpha)} 
\sum_{\begin{subarray}{c} \lambda \in \mcal{P}_n(d) \\
\sqrt{\alpha} \xi_{\lambda_i}^{(n)} \le h_i \ ( 1\le i \le k) \end{subarray}}
\frac{1}{c_{\lambda}(\alpha)c'_{\lambda}(\alpha)} \\
=& \frac{1}{Z_d(2/\alpha)} \int_{\begin{subarray}{c} 
x_1 \ge \cdots \ge x_d \\
x_d := -(x_1 + \cdots +x_{d-1}) \\ x_i \le h_i \ ( 1 \le i \le k) \end{subarray}}
e^{-\frac{1}{\alpha} (x_1^2+ \cdots +x_d^2)} 
\prod_{1 \le i<j \le d}(x_i-x_j)^{2/\alpha} d x_1 \cdots d x_{d-1}. 
\end{align*}
Here the integral runs over $(x_1,\dots,x_{d-1}) \in \bR^{d-1}$ satisfying
$x_1 \ge \cdots \ge x_{d-1} \ge x_d$ and $x_i \le h_i \ (1\le i \le k)$ for 
$x_d:=-(x_1+\cdots+x_{d-1})$.

The rest of the proof of the theorem is similar to the proof of Lemma \ref{lem:SumLimit}. 
We write as
\begin{align*}
& \frac{1}{C_{n,d}(\alpha)} 
\sum_{\begin{subarray}{c} \lambda \in \mcal{P}_n(d) \\
\sqrt{\alpha}\xi_{\lambda_i}^{(n)} \le  h_i \ ( 1\le i \le k) \end{subarray}}
\frac{1}{c_{\lambda}(\alpha)c'_{\lambda}(\alpha)} \\
=& \frac{1}{C_{n,d}(\alpha)} 
\frac{\Gamma(1/\alpha)^d e^{2n}}
{\alpha^{2n} (2\pi)^d \(\frac{n}{d}\)^{2n+\frac{d^2+d}{2\alpha}-\frac{d-1}{2}}}\\
& \times \sum_{\begin{subarray}{c} \lambda \in \mcal{P}_n(d) \\
\sqrt{\alpha} \xi_{\lambda_i}^{(n)} \le h_i \ ( 1\le i \le k) \end{subarray}}
\( \frac{d}{n}\)^{\frac{d-1}{2}} 
\widehat{c}^{(\alpha)} \(\frac{n}{d}+\xi_{\lambda_1}^{(n)}\sqrt{\frac{n}{d}},
\dots, \frac{n}{d}+\xi_{\lambda_d}^{(n)}\sqrt{\frac{n}{d}}\)
\frac{\alpha^{2n} (2\pi)^d \(\frac{n}{d}\)^{2n+\frac{d^2+d}{2\alpha}}}
{\Gamma(1/\alpha)^d e^{2n}}.
\end{align*}
Using \eqref{eq:DensityAsymptotic} and Lemma \ref{lem:SumLimit},
as $n \to \infty$, it converges to 
\begin{align*}
&  \frac{\alpha^{\frac{d(d-1)}{2\alpha}+\frac{d-1}{2}} }{Z_d(2/\alpha)}
\int_{\begin{subarray}{c} 
y_1 \ge \cdots \ge y_d \\
y_d := -(y_1 + \cdots +y_{d-1}) \\ \sqrt{\alpha} y_i \le  h_i \ ( 1 \le i \le k) \end{subarray}}
e^{-(y_1^2+ \cdots +y_d^2)} \prod_{1 \le i<j \le d}(y_i-y_j)^{2/\alpha} d y_1 \cdots d y_{d-1} \\
=& \frac{1}{Z_d(2/\alpha)} \int_{\begin{subarray}{c} 
x_1 \ge \cdots \ge x_d \\
x_d := -(x_1 + \cdots +x_{d-1}) \\ x_i \le h_i \ ( 1 \le i \le k) \end{subarray}}
e^{-\frac{1}{\alpha} (x_1^2+ \cdots +x_d^2)} 
\prod_{1 \le i<j \le d}(x_i-x_j)^{2/\alpha} d x_1 \cdots d x_{d-1}. 
\end{align*}
Thus we have proved Theorem \ref{maintheorem}.

%%%%%%%%%%%%%%%%%%%%%%%%%%%%%%%%%%%%%%%%%%%%%%%%%%%%%%
\section{A short proof of Regev's asymptotic theorem} \label{sec:Regev}
%%%%%%%%%%%%%%%%%%%%%%%%%%%%%%%%%%%%%%%%%%%%%%%%%%%%%%

Applying the technique used in the previous section,
we give a simple proof of 
the following asymptotic theorem by Regev \cite{Regev}.

Let $f^\lambda$ be the degree of the irreducible representation of 
the symmetric group $\mf{S}_{|\lambda|}$ associated with $\lambda$.
Equivalently, $f^\lambda$ is the number of standard Young tableaux of shape $\lambda$.

\begin{thm}[Regev] \label{RegevTheorem}
Let a positive real number $\beta$ and a positive integer $d$ be fixed.
As $n \to \infty$, 
$$
\sum_{\lambda \in \mcal{P}_n(d)} (f^\lambda)^\beta \sim 
\( (2\pi)^{-\frac{d-1}{2}} d^{n+\frac{d^2}{2}}  n^{-\frac{(d-1)(d+2)}{4}} \)^\beta
n^{\frac{d-1}{2}} Z'_d(\beta),
$$
where 
$$
Z'_d(\beta)= \int_{\mathfrak{H}_d} e^{-\frac{d \beta}{2} (x_1^2+ \cdots+x_d^2)}
\prod_{1 \le i<j \le d} (x_i-x_j)^\beta d x_1 \cdots d x_{d-1}.
$$
\end{thm}

\begin{proof}
Recall the hook formula \eqref{eq:hookformula}: $f^\lambda=n!/H_\lambda$.
By Lemma \ref{lem:ccExplicit}, we have
$$
f^\lambda= \frac{n!}{c_{\lambda}(1)}=n! \prod_{1 \le i <j \le d} (\lambda_i-\lambda_j+j-i) \cdot 
\prod_{i=1}^d \frac{1}{\Gamma(\lambda_i+d-i+1)}.
$$
Replacing each $\lambda_i$ by $r_i$ in  this identity, 
we can define the value $f^{(r_1,\dots,r_d)}$ for
all non-negative real numbers $r_1,\dots,r_d$ such that $r_1 \ge \cdots \ge r_d \ge 0$ and 
$r_1+\dots+r_d=n$.
Applying Lemma \ref{lem:DenstiyAsymptotic1},
we see that for each $(y_1,\dots,y_d) \in \mf{H}_d$, putting
$r_i=\frac{n}{d}+y_i \sqrt{\frac{n}{d}} \ (1 \le i \le d)$,
$$
\lim_{n\to\infty}
\frac{\(\frac{n}{d}\)^{n+\frac{d^2+d}{4}}}{n! e^n} 
f^{(r_1,\dots,r_d) }
=\prod_{1 \le i<j \le d} (y_i-y_j)
\cdot \prod_{i=1}^d \frac{e^{-\frac{1}{2}y_i^2} }{\sqrt{2\pi}}.
$$
Furthermore, by the  
Stirling formula $n! \sim \sqrt{2\pi} n^{n+\frac{1}{2}}e^{-n}$, we have
$$
\lim_{n \to\infty} (2\pi)^{\frac{d-1}{2}}
d^{-n-\frac{d^2+d}{4}}  n^{\frac{d^2+d}{4}-\frac{1}{2}}
f^{(r_1,\dots,r_d)} =  
\prod_{1 \le i<j \le d} (y_i-y_j)
\cdot \prod_{i=1}^d e^{-\frac{1}{2}y_i^2}.
$$
On the other hand, as in the proof of \eqref{eq:Dominated}, we can prove that
there exists a polynomial $P$ such that 
$$
d^{-n-\frac{d^2+d}{4}}  n^{\frac{d^2+d}{4}-\frac{1}{2}}
f^{(r_1,\dots,r_d)} 
< P(y_1,\dots,y_d) e^{-|y_1|-\cdots-|y_d|}
$$
for all $n$ and $y_1,\dots,y_d$.

Consider 
$$
\((2\pi)^{\frac{d-1}{2}} d^{-n-\frac{d^2+d}{4}} n^{ \frac{(d+2)(d-1)}{4}} \)^\beta
\(\frac{d}{n}\)^{\frac{d-1}{2}}
\sum_{\lambda \in \mcal{P}_n(d)}  (f^\lambda)^\beta.
$$
As in the proof of Lemma \ref{lem:SumLimit}, 
we can apply the dominated convergence theorem 
with a dominated function of the form $P(y_1,\dots,y_d)^\beta e^{-\beta(|y_1|+\cdots+|y_d|)}$.
Then
\begin{align*}
& \lim_{n \to \infty} 
\((2\pi)^{\frac{d-1}{2}} d^{-n-\frac{d^2+d}{4}} n^{ \frac{(d+2)(d-1)}{4}} \)^\beta
\sum_{\lambda \in \mcal{P}_n(d)} \(\frac{d}{n}\)^{\frac{d-1}{2}} (f^\lambda)^\beta \\
=& 
\int_{\mf{H}_d} e^{-\frac{\beta}{2} (y_1^2+\cdots+y_d^2)} \prod_{1 \le i<j \le d}
(y_i-y_j)^\beta  d y_1 \cdots d y_{d-1} 
=  Z_d(\beta).
\end{align*}
It is easy to see that 
$Z_d(\beta)= d^{\frac{d(d-1)}{4} \beta + \frac{d-1}{2}} Z_d'(\beta)$.
Thus, we have proved the theorem.
\end{proof}

Using Lemma \ref{lem:SumLimit} and the Stirling formula, we obtain the following.

\begin{prop}
As $n \to \infty$,
$$
\sum_{\lambda \in \mcal{P}_n(d)} \frac{n!}{c_{\lambda}(\alpha)} \frac{n!}{c'_{\lambda}(\alpha)} \\
\sim  \frac{\Gamma(1/\alpha)^{d} d^{2n+\frac{d^2}{\alpha}}}
{(2\pi)^{d-1} \alpha^{2n+\frac{d^2-d}{2\alpha}+\frac{d-1}{2}} 
n^{\frac{d^2+d}{2\alpha}-\frac{d+1}{2}} }
Z_d'(2/\alpha). 
$$
\end{prop}

This proposition may be seen as a Jack analogue (of the $\beta=2$ case) of
Theorem \ref{RegevTheorem}.

\medskip
\medskip

\noindent
{\bf Acknowledgements}

\noindent
This research was supported by the Grant-in-Aid for JSPS Fellows.

%<<<<<<<<<<<<<<<<<<<<<<<<<<<<<<<<<<<<<<<<<


\begin{thebibliography}{AAAA}
%>>>>>>>>>>>>>>>>>>>>>>>>>>>>>>>>>>>>>>>>>
%
%
{\small

\bibitem[AAR]{AAR}
 G. E. Andrews, R. Askey, and R. Roy,
 ``Special Functions'', 
 Encyclopedia Math. Appl. 71, Cambridge Univ. Press, Cambridge, 1999.


\bibitem[BDJ]{BDJ}
 J. Baik, P. Deift, and K. Johansson,
 On the distribution of the length of the longest increasing subsequence of random permutations,
 J. Amer. Math. Soc. {\bf 12} (1999), no. 4,  1119--1178.

\bibitem[BR1]{BR1}
 J. Baik and E. M. Rains,
 Algebraic aspects of increasing subsequences,
 Duke Math. J. {\bf 109} (2001), no. 1, 1--65.

\bibitem[BR2]{BR2}
 J. Baik and E. M. Rains,
 The asymptotics of monotone subsequences of involutions,
 Duke Math. J. {\bf 109} (2001), no. 2, 205--281.

\bibitem[BR3]{BR3}
 J. Baik and E. M. Rains,
 Symmetrized random permutations, 
 random matrix models and their applications, 1--19, 
 Math. Sci. Res. Inst. Publ., 40, Cambridge Univ. Press, Cambridge, 2001. 

\bibitem[B]{Borodin}
 A. Borodin,
 Longest increasing subsequences of random colored permutations, 
 Electron. J. Combin. {\bf 6} (1999), Research Paper 13, 12 pages. 


\bibitem[BOO]{BOO}
 A. Borodin, A. Okounkov, and G. Olshanski,
 Asymptotics of Plancherel measures for symmetric groups,
 J. Amer. Math. Soc. {\bf 13} (2000), 481--515.

\bibitem[BO]{BorodinOlshanski}
 A. Borodin and G. Olshanski,
 Z-measures on partitions and their scaling limits,
 European J. Combi. {\bf 26} (2005), 795--834.

\bibitem[BOS]{BorodinOlshanskiStrahov}
 A. Borodin, G. Olshanski, and E. Strahov,
 Giambelli compatible point processes,
 Adv. in Appl. Math. {\bf 37} (2006), no. 2, 209--248.

\bibitem[DE]{DE}
 I. Dumitriu and A. Edelman,
 Matrix models for beta ensembles,
 J. Math. Phys. {\bf 43} (2002), no. 11, 5830--5847.
 
\bibitem[Fo]{Forrester}
 P. J. Forrester, 
 ``Log-gases and Random matrices'', book in preparations,
 \verb|http://www.ms.unimelb.edu.au/~matpjf/matpjf.html|.

\bibitem[FNR]{FNR}
 P. J. Forrester, T. Nagao, and E. M. Rains,
 Correlation functions for random involutions,
 Internat. Math. Res. Notices (2006), articleID 89796, 1--35.

\bibitem[FR]{FR}
 P. J. Forrester and E. M. Rains,
 Symmetrized models of last passage percolation and non-intersecting lattice paths, 
 J. Stat. Phys. {\bf 129} (2007), no. 5-6, 833--855. 

\bibitem[Fu1]{FulmanJack}
 J. Fulman,
 Stein's method, Jack measure, and the Metropolis algorithm,
 J. Combin. Theory Ser. A {\bf 108} (2004),  275--296.

\bibitem[Fu2]{FulmanCharacter}
 J. Fulman,
 Stein's method and random character ratios,
 Trans. Amer. Math. Soc. {\bf 360} (2008), no. 7, 3687--3730.

\bibitem[J1]{JohanssonUnitary}
 K. Johansson,
 The longest increasing subsequence in a random permutation and a unitary random matrix model, 
 Math. Res. Lett. {\bf 5} (1998), no. 1-2, 63--82. 

\bibitem[J2]{JohanssonShape}
 K. Johansson,
 Shape fluctuations and random matrices,
 Comm. Math. Phys. {\bf 209} (2000), no. 2, 437--476.

\bibitem[J3]{Johansson}
 K. Johansson,
 Discrete orthogonal polynomial ensembles and the Plancherel measure,
 Ann. Math. (2) {\bf 153} (2001), 259--296.

\bibitem[K]{KerovBook}
 S. V. Kerov, 
 ``Asymptotic representation theory of the symmetric group and its applications in analysis'', 
 Translated from the Russian manuscript by N. V. Tsilevich. 
 Translations of Mathematical Monographs, {\bf 219}. 
 American Mathematical Society, Providence, RI, 2003.

\bibitem[Mac]{Mac}
 I. G. Macdonald,
 ``Symmetric Functions and Hall Polynomials'', 2nd ed.,
 Oxford University Press,
 Oxford, 1995.

\bibitem[Mat1]{MatSS1}
 S. Matsumoto, 
 Correlation functions of the shifted Schur measure,
 J. Math. Soc. Japan {\bf 57} (2005), 619--637.

\bibitem[Mat2]{MatSS2}
 S. Matsumoto, $\alpha$-Pfaffian, pfaffian point process and shifted Schur measure,
 Linear Alg. Appl. {\bf 403} (2005), 369--398. 

\bibitem[Me]{Mehta}
 M. L. Mehta, 
 ``Random matrices'', 3rd ed., 
 Pure and Applied Mathematics (Amsterdam), 142, Elsevier/Academic Press, Amsterdam, 2004. 
 
\bibitem[O1]{Okounkov1}
 A. Okounkov,
 Random matrices and random permutations,
 Internat. Math. Res. Notices (2000), 1043--1095.

\bibitem[O2]{Okounkov2}
 A. Okounkov,
 The uses of random partitions,
 XIVth International Congress on Mathematical Physics, 379--403, World Sci. Publ.,
 Hackensack, NJ 2005, available: arXiv:math-ph/0309015v1. 

\bibitem[Ra]{Rains}
 E. M. Rains,
 Increasing subsequences and the classical groups,
 Electron. J. Combin. {\bf 5} (1998), Research Paper 12, 9 pages.

\bibitem[Re]{Regev}
 A. Regev,
 Asymptotic values for degrees associated with strips of Young diagrams,
 Adv. in Math. {\bf 41} (1981), 115--136.

\bibitem[Sa]{Sagan}
 B. E. Sagan, 
 ``The symmetric group. Representations, combinatorial algorithms, and symmetric functions'', 
 2nd ed., Graduate Texts in Mathematics, {\bf 203}, Springer-Verlag, New York, 2001. 

\bibitem[Sn]{Sniady}
 P. \'{S}niady, 
 Permutations without long decreasing subsequences and random matrices,
 Electron. J. Combin. {\bf 14} (2007), no. 1, Research Paper 11, 9 pages. 

\bibitem[St]{Strahov}
 E. Strahov,
 Matrix kernels for measures on partitions,
 arXiv:0804.2419v2, 25 pages.

\bibitem[TW1]{TWlevel}
 C. A. Tracy and H. Widom,
 Level-spacing distributions and the Airy kernel, 
 Comm. Math. Phys. {\bf 159} (1994), no. 1, 151--174.

\bibitem[TW2]{TW1}
 C. A. Tracy and H. Widom,
 On the distributions of the lengths of the longest monotone subsequences in random words, 
 Probab. Theory Related Fields {\bf 119} (2001), no. 3, 350--380.

\bibitem[TW3]{TW2}
 C. A. Tracy and H. Widom,
 A limit theorem for shifted Schur measures, 
 Duke Math. J. {\bf 123} (2004), no. 1, 171--208.


}
\end{thebibliography}
\end{document}